\pdfoutput=1
\documentclass[reqno]{amsart}
\usepackage{float}
\usepackage{graphicx} 
\usepackage{hyperref}

\usepackage{amsmath, amssymb, amsthm}
\usepackage{comment}
\theoremstyle{theorem}
\newtheorem{thm}{Theorem}[section]
\newtheorem{lemma}[thm]{Lemma}
\newtheorem{prop}[thm]{Proposition}
\newtheorem{cor}[thm]{Corollary}

\theoremstyle{definition}
\newtheorem{defn}[thm]{Definition}
\newtheorem{notation}[thm]{Notation}

\newcommand{\bm}[1]{{\mbox{\boldmath $#1$}}}

\theoremstyle{remark}
\newtheorem{remark}{Remark}

\newcommand{\III}{\textup{I\!I\!I}}
\newcommand{\II}{\textup{I\!I}}

\newcommand{\RIpos}{\begin{picture}(12,0)
\put(1,-2){\includegraphics{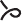}}
\end{picture}}
\newcommand{\RIneg}{\begin{picture}(12,0)
\put(1,-2){\includegraphics{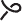}}
\end{picture}}
\newcommand{\RIm}{\begin{picture}(12,0)
\put(2,-2){\includegraphics{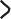}}
\end{picture}}

\title[B-type coefficient polynomial]{B-type coefficient polynomial}
\author{Noboru Ito}
\address{Department of Mathematics, Faculty of Engineering, Shinshu University, Wakasato 4-17-1  Nagano, Nagano,  380-8553, Japan}
\email{nito@shinshu-u.ac.jp}
\author{Mayuko Kon}
\address{Faculty of Education, Shinshu University, 6-Ro, Nishinagano, Nagano City 380-8544, Japan}
\email{mayuko{\_}k@shinshu-u.ac.jp}
\date{March 26, 2026}

\begin{document}
\begin{abstract}
An A-type coefficient polynomial introduced by Kawauchi \cite{Kawauchi1994}
recovers the HOMFLY--PT polynomial as a formal power series 
within skein theory.
A notable feature of this construction is that each coefficient defines
a link invariant, yielding an infinite sequence of invariants,
while the low-degree coefficients are relatively easy to compute.
In this paper, we extend this viewpoint to the B-type setting.
Unlike the A-type case, the B-type setting requires a genuinely new inductive scheme due to the four-term skein relation.  
More precisely, we introduce coefficient polynomials associated with
the B-type skein relation and show that their generating series
recovers the Kauffman polynomial.
We further prove that these coefficient polynomials are well-defined
and that the resulting generating series is invariant under the
corresponding Reidemeister moves.
\end{abstract}
\maketitle

\section{Introduction}
Skein theory has played a fundamental role in low-dimensional topology
as a framework for treating link invariants in a unified manner.
The Jones polynomial and the HOMFLY--PT polynomial
\cite{HOMFLY1985,PrzytyckiTraczyk1987}
are representative examples,
both of which can be characterized by skein relations 
together with normalization conditions.
The Kauffman polynomial \cite{Kauffman1990}
is characterized by a four-term skein relation together with
regular isotopy conditions.

The A-type coefficient polynomial introduced by Kawauchi
\cite{Kawauchi1994}
provides a skein-theoretic construction that recovers the
HOMFLY--PT polynomial as a formal power series.
In this theory, each coefficient can be interpreted as a link invariant,
yielding an infinite sequence of invariants,
and in particular the low-degree coefficients are relatively easy to compute.
From this viewpoint, the A-type coefficient polynomial  may be regarded as a
coefficient-level refinement of the HOMFLY--PT polynomial.

The purpose of this paper is to show that this coefficient-level
construction can be extended to the B-type setting.
More precisely, following the same formal power series framework as in
Kawauchi's construction, we introduce coefficient polynomials
associated with the B-type skein relation and show that their
generating series recovers the Kauffman polynomial.
Thus, the Kauffman polynomial is reconstructed from a family of
coefficient invariants arising from skein theory.

A central point is that the B-type case is not a formal repetition of
the A-type one.
In the A-type case, the well-definedness is based on a three-term skein
relation together with an inductive argument on monotone diagrams.
In contrast, in the B-type case, one first needs a new coefficient-level
formulation of a four-term skein relation, and the same inductive
mechanism does not apply directly.

We overcome this difficulty by refining the initial conditions for the
inductive construction and by introducing a bookkeeping of the change
in the number of components under the two splice operations.
This reflects a structural difference between three-term and four-term
skein relations:
while the inductive construction in the A-type case closes within the
class of monotone diagrams,
the four-term relation naturally produces connected sums of monotone
diagrams.
As a consequence, Kawauchi's method extends to the B-type setting only
after incorporating this new phenomenon into the induction scheme.

As a result, we obtain a coefficient-level extension of Kawauchi's
method to the B-type setting, providing a foundation for further
systematic studies of coefficient invariants associated with
four-term skein relations.

The organization of this paper is as follows.
In Section~\ref{sec:prelim} we introduce the notation and
diagrammatic conventions used throughout the paper.
In Section~\ref{sec:mainresult} we state the main theorem describing
the coefficient polynomials $\alpha_n(D;y)$.
Section~\ref{sec:construction} is devoted to the construction of
$\alpha_n(D;y)$ and the proof of the main theorem using an inductive
argument based on warping crossings.
In Section~\ref{sec:uniqueness} we prove uniqueness for the coefficient
polynomials and for the resulting generating series, and show that the
latter recovers the Kauffman polynomial.
Finally, Section~\ref{sec:discussion} contains remarks on the structure
of the coefficient polynomials and directions for further study.

\section{Preliminaries and notation}\label{sec:prelim}
Throughout this paper, diagrams are considered up to planar isotopy.
We fix the conventions and notation used in the construction
of the coefficient polynomials associated with the B-type skein relation.

\subsection{Link diagrams and crossings}

\begin{defn}[link diagrams]
An \emph{unoriented link diagram} $D$ is a generic immersion
of a disjoint union of circles in the plane with over/under information at each transverse double point; 
an \emph{oriented link diagram} is such a diagram with
an orientation on each component.
\end{defn}

\begin{notation}[$c(D)$, $r(D)$]
Let $c(D)$ be the crossing number of $D$ and $r(D)$ 
the number of components of $D$.
\end{notation}

\begin{defn}[writhe]
Given an oriented diagram $D$, each crossing $p$ has a sign
$\epsilon(p)\in\{+1,-1\}$.
The \emph{writhe} is defined by 
\[
w(D)=\sum_{p~:~ {\textrm{a crossing of $D$}}} \epsilon(p).
\]
\end{defn}

\subsection{Local splices}

\begin{figure}[htbp]
\centering
\begin{picture}(30,30)
\put(0,0){\line(1,1){30}}
\put(30,0){\line(-1,1){10}}
\put(0,30){\line(1,-1){10}}
\put(10,-10){$D_p$}
\end{picture}
\qquad
\begin{picture}(30,30)
\put(10,-10){$D_{\infty}$}
\qbezier(0,0)(10,15)(0,30)
\qbezier(30,30)(20,15)(30,0)
\end{picture}
\qquad
\begin{picture}(30,30)
\put(10,-10){$D_0$}
\qbezier(0,0)(15,10)(30,0)
\qbezier(30,30)(15,20)(0,30)
\end{picture}
\caption{A crossing $D_p$ and its two smoothings
$D_\infty$ and $D_0$.}
\label{fig:splices}
\end{figure}
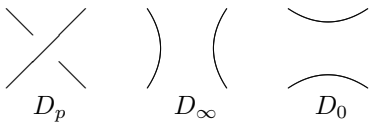

In Figure~\ref{fig:splices}, we denote by $D_\infty$ and $D_0$
the two smoothings of $D$ at a crossing $p$ shown there.
We call the smoothing producing $D_\infty$
(resp.\ $D_0$) the $A$-splice (resp.\ $B$-splice).
Accordingly, we set
\[
D_{p1}=D_{\infty}, \qquad
D_{p2}=D_{0}.
\]


\begin{notation}[Four local diagrams]
We use the notation $(D_+,D_-,D_\infty, D_0)$
for the four standard local diagrams at a crossing.
The skein relation is formulated for unoriented diagrams.

For the purpose of proofs, we may temporarily equip
the strands near $p$ with an auxiliary orientation.
This determines a sign $\epsilon(p)$ and fixes an ordered pair
$
(D_{\epsilon(p)},D_{-\epsilon(p)})
$ which is either 
$(D_+,D_-)$ or $(D_-,D_+)$.
This auxiliary choice is local and fixed within each argument.
\end{notation}

\begin{remark}
For notational convenience, we use the symbols
$p1$ and $p2$ to distinguish the two splices; thus
$D_{p1}=D_\infty$ and $D_{p2}=D_0$.
\end{remark}

\subsection{Component-change quantities}

\begin{defn}[Component indicator]\label{def:delta}
Let $p$ be a crossing of $D$.
Define
\[
\delta(p)=
\begin{cases}
0 & \text{if the strands at $p$ belong to the same component},\\
1 & \text{otherwise}.
\end{cases}
\]
\end{defn}

\begin{defn}[Signed component change]\label{def:Delta}
For the spliced diagrams $D_{p1}$ and $D_{p2}$,
define
\[
\Delta(p1)=r(D_{p1})-r(D),\qquad
\Delta(p2)=r(D_{p2})-r(D).
\]
Thus $\Delta(pi)$ records the signed change in the number of
components under the corresponding splice.
\end{defn}

\begin{remark}
In the inductive construction, the coefficient index shifts
according to
\[
n\mapsto n+\Delta(pi)-1,
\]
and successive splices produce additive shifts such as
\[
n+\Delta(p1)+\Delta(q1)-2.
\]
This shift is essential in formulating the coefficient-level skein relation.
\end{remark}

\subsection{Algebraic conventions}

\begin{notation}\label{not:ringR}
Fix a commutative ring $R$ containing $\mathbb{Z}[y^{\pm1}]$ as a subring.  We write $R[[z]]$ for the ring of formal power series in $z$ with coefficients in $R$.
\end{notation}

\subsection{Coefficient polynomials}

\begin{defn}\label{def:alpha_in_R}
For an unoriented diagram $D$,
we define elements
\[
\alpha_n(D;y)\in R
\qquad (n\in\mathbb{Z}).
\]
\end{defn}

\begin{defn}[Generating series of the coefficient polynomials]\label{def:L_in_R}
Let $D$ be an unoriented diagram.
Define
\[
L_D(y,z)=z^{\,1-r(D)}\sum_{n\ge0}\alpha_n(D;y)z^n \in R[[z]].
\]
\end{defn}

\begin{defn}[Normalization]
For an oriented diagram $D$, define
\[
F_D(y,z)=y^{-w(D)}L_D(y,z).
\]
\end{defn}

\subsection{Base points and warping degree}

\begin{defn}[Sequences of base points and induced directions]\label{def:basepoints}
Let $D$ be an unoriented diagram with $r(D)=r$ components.
A \emph{sequence of base points} is an ordered $r$-tuple
$\bm{a}=(a_1,\dots,a_r)$ with one point $a_i$ on each component.
For each component, we additionally choose a direction of travel starting from $a_i$.
(Equivalently, we choose an orientation on each component; this choice is used only to define warping data.)
\end{defn}

\begin{defn}[Connected sequences of base points]\label{def:connected_basepoints}
Two sequences of base points $\bm{a}=(a_1,\dots,a_r)$ and $\bm{a}'=(a'_1,\dots,a'_r)$ on a diagram $D$ are said to be \emph{connected} if $a_i$ and $a'_i$ belong to the same connected component of $D$ for each $i=1, \dots, r$.
\end{defn}

\begin{defn}[First-encounter rule at a crossing]\label{def:firstencounter}
Fix a based-and-directed diagram $(D,\bm{a})$ as in Definition~\ref{def:basepoints}.
For a crossing $p$ of $D$, traverse each component once starting from its base point in the chosen direction.
Among the two strands meeting at $p$, exactly one of them is encountered first in this traversal.
We call it the \emph{first-encountered strand} at $p$ (with respect to $(D,\bm{a})$).
\end{defn}

\begin{defn}[Warping crossing and warping degree]\label{def:warping}
A crossing $p$ of $D$ is called a \emph{warping crossing point} of $(D,\bm{a})$ 
if the first-encountered strand at $p$ is the \emph{under}-strand of the crossing.
The \emph{warping degree} $d_\bm{a}(D)$ is the number of warping crossing points of $(D,\bm{a})$.
\end{defn}

\begin{defn}[Monotone diagram]\label{def:monotone}
A based-and-directed diagram $(D,\bm{a})$ is called \emph{monotone}
if $d_\bm{a}(D)=0$.
\end{defn}
\begin{notation}[Complexity]
We write
\[
cd(D) = \bigl(c(D), d_{\bm a}(D)\bigr)
\]
for the pair consisting of the crossing number and the warping degree,
which will be used as a complexity measure in the inductive arguments.
\end{notation}




\subsection{Disjoint union and connected sum}

\begin{notation}
Let $O$ denote the zero-crossing knot diagram.
We denote disjoint union by $\sqcup$ and may write $O+D$
for $O\sqcup D$.
\end{notation}

\begin{notation}
We write $D\#D'$ for a connected sum of unoriented diagrams.
When the orientation of $D\#D'$ is needed, it will be specified after summing.    
\end{notation}

\section{Main Result}\label{sec:mainresult}
\begin{thm}\label{maintheorem1}
For every unoriented link diagram $D$, there exists a sequence of coefficient polynomials
\[
\alpha_n(D;y)\in R \qquad (n\in\mathbb Z)
\]
such that the following properties hold.

\begin{itemize}
    \item[$(1)$] $\alpha_{n}(D; y)$ is invariant under the Reidemeister moves of type \II~and type \III, and transforms under the Reidemeister move of type~I as in \eqref{equ1}: 
\begin{align}\label{equ1}
   & \alpha_{n}\left( 
   \RIpos
   ; y\right) = y \alpha_{n}\left(\RIm; y\right),\\
   & \alpha_{n}\left(\RIneg; y\right) = y^{-1} \alpha_{n}\left(\RIm; y \right).\nonumber
 \end{align}
        \item[$(2)$] For the zero-crossing knot diagram $O$,
\begin{equation}\label{equ2}
    \alpha_{n}(O; y) = \delta_{n, 0},
\end{equation}
where $\delta_{n,0}$ denotes the Kronecker delta.
    \item[$(3)$] For $(D_+, D_-, D_{p1}, D_{p2})$ as in Figure~\ref{fig:splices},
\begin{eqnarray}\label{equ3} 
& &\alpha_{n}(D_+; y) + \alpha_{n}(D_-; y) \\
& &= \alpha_{n+\Delta(p1)-1}(D_{p1}; y) + \alpha_{n+\Delta(p2)-1}(D_{p2}; y).\nonumber
\end{eqnarray}
Here, $\Delta(p1)$ and $\Delta(p2)$ are as in Definition~\ref{def:Delta}.  
\end{itemize}
\end{thm}

\section{Construction and Proof of Theorem~\ref{maintheorem1}}\label{sec:construction}
Before proceeding to the proof, we briefly explain the main difference
from the A-type case.
In the A-type construction, the skein relation consists of three terms,
and the inductive argument closes within the class of monotone diagrams.
In contrast, the B-type skein relation involves four terms.
When comparing the expansions at two distinct warping crossings,
the splice terms naturally produce intermediate diagrams
which may decompose as connected sums of monotone diagrams.
For this reason, in proofs of this section, we enlarge the initial class
to include connected sums of monotone diagrams.
Apart from this point, the inductive structure follows the same strategy
as in the A-type case.
The necessity of this enlargement is intrinsic to the four-term nature
of the B-type skein relation and does not arise in the three-term A-type case.

To prove the theorem, we use induction on the crossing number $c(D)=m$. 

When $m=0$, we set
\[
\alpha_{n}(D, \bm{a}; y) = (-1)^{n} \binom{r-1}{n} (y + y^{-1})^{r-n-1},
\]
where $\bm{a}$ is a sequence of base points and $r$ is the number of connected components. Then any diagram $D$ with $m=0$ satisfies properties (1)--(3) in Theorem~\ref{maintheorem1}.

We suppose that for any diagram $D$ with $c(D)<m$, there exists $\alpha_{n}(D; y)$ which satisfies properties (1)--(3) in Theorem~\ref{maintheorem1}. We show that we can construct $\alpha_{n}(D; y)$ for $D$ with $c(D)=m$ by the following steps.
\begin{itemize}
    \item[(i)] For a pair $(D, \bm{a})$ of a diagram and a sequence of base points with $c(D)=m$, we define $\alpha_{n}(D, \bm{a}; y)$.
    \item[(ii)] We show that $\alpha_{n}(D, \bm{a}; y)$ does not depend on the choice of the sequence of base points $\bm{a}$. We denote this by $\alpha_{n}(D; y)$.
    \item[(iii)] For any diagram $D$ with $c(D) \le m$, $\alpha_{n}(D; y)$ satisfies properties (1)--(3).
\end{itemize}

We define $\alpha_{n}(D, \bm{a}; y)$ as follows. 

\begin{defn}\label{defofalpha}
When $D$ is an unoriented monotone diagram or a connected sum of monotone diagrams, we set
\[
\alpha_{n}(D, \bm{a}; y) = y^{w(D)} (-1)^{n} \binom{r-1}{n} (y + y^{-1})^{r-n-1}.
\]
When $s=d_{\bm{a}}(D)>0$, we define $\alpha_{n}(D, \bm{a}; y)$ by induction:
\begin{align}\label{equdef}
    \alpha_{n}(D, \bm{a}; y) &= \alpha_{n}(D_{\epsilon(p)}, \bm{a}; y) \\
    &= -\alpha_{n}(D_{-\epsilon(p)}, \bm{a}; y) + \alpha_{n+\Delta(p1)-1}(D_{p1}; y) + \alpha_{n+\Delta(p2)-1}(D_{p2}; y). \nonumber
\end{align}
Here $\Delta(p1)$ and $\Delta(p2)$ are as in Definition~\ref{def:Delta}.
\end{defn}

\begin{lemma}\label{warpingcrosspt}
Let $D$ be a diagram with $c(D)=m$. When $s>0$, $\alpha_{n}(D, \bm{a}; y)$ does not depend on the choice of the warping crossing point $p$ of $(D, \bm{a})$.
\end{lemma}

\begin{proof}
Let $q \neq p$ be another warping crossing point of $(D, \bm{a})$. We will show the following equality:
\begin{align*}
&-\alpha_{n}(D_{-\epsilon(p)}, \bm{a}; y) + \alpha_{n+\Delta(p1)-1}(D_{p1}; y) + \alpha_{n+\Delta(p2)-1}(D_{p2}; y)\\
&= - \alpha_{n}(D_{-\epsilon(q)}, \bm{a}; y) + \alpha_{n+\Delta(q1)-1}(D_{q1}; y) + \alpha_{n+\Delta(q2)-1}(D_{q2}; y).
\end{align*}

Since $d_{\bm{a}}(D_{-\epsilon(p)}) < s$ and $q$ is a warping crossing point of $(D_{-\epsilon(p)}, \bm{a})$, by the induction hypothesis on $s$, we have
\begin{align*}
\alpha_{n}(D_{-\epsilon(p)}, \bm{a}; y) &= -\alpha_{n}((D_{-\epsilon(p)})_{-\epsilon(q)}, \bm{a}; y) \\
&\quad + \alpha_{n+\Delta(q1)-1}((D_{-\epsilon(p)})_{q1}; y) + \alpha_{n+\Delta(q2)-1}((D_{-\epsilon(p)})_{q2}; y).
\end{align*}
Next, since $c(D_{p1}) < m$ and $c((D_{p1})_{-\epsilon(q)}) < m$, by the induction hypothesis on the number of crossings, we have
\begin{align*}
&\alpha_{n+\Delta(p1)-1}(D_{p1}; y) + \alpha_{n+\Delta(p1)-1}((D_{p1})_{-\epsilon(q)}; y) \\
&\quad = \alpha_{n+\Delta(p1)+\Delta'(q1)-2}((D_{p1})_{q1}; y) + \alpha_{n+\Delta(p1)+\Delta'(q2)-2}((D_{p1})_{q2}; y).
\end{align*}
Here $\Delta'(q1)$ (resp.~$\Delta'(q2)$)
denotes the signed change in the number of components
under the corresponding splice at $q$
in the diagram $D_{p1}$; namely,
\[
\Delta'(qi)=r\big((D_{p1})_{qi}\big)-r(D_{p1})
\quad (i=1,2).
\]
Similarly, we have
\begin{align*}
&\alpha_{n+\Delta(p2)-1}(D_{p2}; y) + \alpha_{n+\Delta(p2)-1}((D_{p2})_{-\epsilon(q)}; y) \\
&\quad = \alpha_{n+\Delta(p2)+\Delta'(q1)-2}((D_{p2})_{q1}; y) + \alpha_{n+\Delta(p2)+\Delta'(q2)-2}((D_{p2})_{q2}; y).  
\end{align*}
From these equations, we obtain
\begin{align*}
&-\alpha_{n}(D_{-\epsilon(p)}, \bm{a}; y) + \alpha_{n+\Delta(p1)-1}(D_{p1}; y) + \alpha_{n+\Delta(p2)-1}(D_{p2}; y)\\
&\quad =\alpha_{n}((D_{-\epsilon(p)})_{-\epsilon(q)}, \bm{a}; y) \\
&\quad\quad - \alpha_{n+\Delta(q1)-1}((D_{-\epsilon(p)})_{q1}; y) - \alpha_{n+\Delta(q2)-1}((D_{-\epsilon(p)})_{q2}; y) \\
&\quad\quad - \alpha_{n+\Delta(p1)-1}((D_{p1})_{-\epsilon(q)}; y) + \alpha_{n+\Delta(p1)+\Delta'(q1)-2}((D_{p1})_{q1}; y) \\
&\quad\quad + \alpha_{n+\Delta(p1)+\Delta'(q2)-2}((D_{p1})_{q2}; y) - \alpha_{n+\Delta(p2)-1}((D_{p2})_{-\epsilon(q)}; y) \\
&\quad\quad + \alpha_{n+\Delta(p2)+\Delta'(q1)-2}((D_{p2})_{q1}; y) + \alpha_{n+\Delta(p2)+\Delta'(q2)-2}((D_{p2})_{q2}; y).
\end{align*}
Similarly, we have
\begin{align*}
&- \alpha_{n}(D_{-\epsilon(q)}, \bm{a}; y) + \alpha_{n+\Delta(q1)-1}(D_{q1}; y) + \alpha_{n+\Delta(q2)-1}(D_{q2}; y)\\
&\quad = \alpha_{n}((D_{-\epsilon(q)})_{-\epsilon(p)}, \bm{a}; y) \\
&\quad\quad - \alpha_{n+\Delta(p1)-1}((D_{-\epsilon(q)})_{p1}; y) - \alpha_{n+\Delta(p2)-1}((D_{-\epsilon(q)})_{p2}; y) \\
&\quad\quad - \alpha_{n+\Delta(q1)-1}((D_{q1})_{-\epsilon(p)}; y) + \alpha_{n+\Delta(q1)+\Delta'(p1)-2}((D_{q1})_{p1}; y) \\
&\quad\quad + \alpha_{n+\Delta(q1)+\Delta'(p2)-2}((D_{q1})_{p2}; y) - \alpha_{n+\Delta(q2)-1}((D_{q2})_{-\epsilon(p)}; y) \\
&\quad\quad + \alpha_{n+\Delta(q2)+\Delta'(p1)-2}((D_{q2})_{p1}; y) + \alpha_{n+\Delta(q2)+\Delta'(p2)-2}((D_{q2})_{p2}; y). 
\end{align*}
Note that the total change in the number of components is independent of the order of the splices. Thus, we obtain the desired equality.
\end{proof}

\begin{lemma}\label{basepoints}
For any diagram $D$ with $c(D)=m$ and any sequence of base points $\bm{a}$, $\alpha_n(D,\bm{a};y)$ does not depend on the choice of a sequence of base points that is connected to $\bm{a}$.
\end{lemma}

\begin{proof}
It is sufficient to show $\alpha_n(D,\bm{a};y)=\alpha_n(D,\bm{a}'; y)$ for the sequences of base points
\begin{align*}
\bm{a}  &= (a_1, \dots, a_{i-1}, a_i, a_{i+1}, \dots, a_r), \\
\bm{a}' &= (a_1, \dots, a_{i-1}, a'_i, a_{i+1}, \dots, a_r)
\end{align*}
as shown in Figure~\ref{WarpingCross}.

\begin{figure}[htbp]
    \centering
\includegraphics[width=0.8\linewidth]{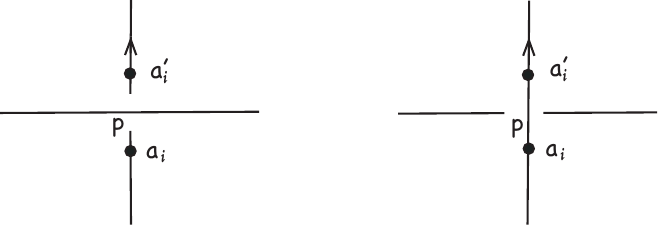}
    \caption{A connected change of the sequence of base points. The base point $a_i$ moves along the component across the crossing $p$ to $a'_i$.}
    \label{WarpingCross}
\end{figure}

Suppose that $q \neq p$ is another warping crossing point. Then we have
\begin{align*}
\alpha_{n}(D, \bm{a}; y) + \alpha_{n}(D_{-\epsilon(q)}, \bm{a}; y) &= \alpha_{n+\Delta(q1)-1}(D_{q1}; y) + \alpha_{n+\Delta(q2)-1}(D_{q2}; y), \\
\alpha_{n}(D, \bm{a}'; y) + \alpha_{n}(D_{-\epsilon(q)}, \bm{a}'; y) &= \alpha_{n+\Delta(q1)-1}(D_{q1}; y) + \alpha_{n+\Delta(q2)-1}(D_{q2}; y).
\end{align*}
By the induction hypothesis on the number of crossing points, we obtain
\[
\alpha_{n}(D, \bm{a}; y) - \alpha_{n}(D, \bm{a}'; y) = \alpha_{n}(D_{-\epsilon(q)}, \bm{a}'; y) - \alpha_{n}(D_{-\epsilon(q)}, \bm{a}; y).
\]

By applying crossing changes away from $p$ if necessary, the diagram $D$ becomes one of the following:
\begin{itemize}
    \item $\delta(p)=1$ and $d_{\bm{a}}(D)=d_{\bm{a}'}(D)=1$, or $\delta(p)=1$ and $d_{\bm{a}}(D)=d_{\bm{a}'}(D)=0$;
    \item $\delta(p)=0$ and $d_{\bm{a}}(D)=1, d_{\bm{a}'}(D)=0$;
    \item $\delta(p)=0$ and $d_{\bm{a}}(D)=0, d_{\bm{a}'}(D)=1$.
\end{itemize}

When $d_{\bm{a}}(D)=d_{\bm{a}'}(D)=0$, we have
\[
\alpha_{n}(D, \bm{a}; y) = y^{w(D)}(-1)^n \binom{r-1}{n} (y+y^{-1})^{r-n-1} = \alpha_{n}(D, \bm{a}'; y).
\]

If $d_{\bm{a}}(D)=d_{\bm{a}'}(D)=1$, then the diagram $D_{-\epsilon(p)}$ is a monotone diagram. Thus, we obtain
\[
\alpha_{n}(D_{-\epsilon(p)}, \bm{a}; y) = \alpha_{n}(D_{-\epsilon(p)}, \bm{a}'; y).
\]
Therefore, we have
\begin{align*}
\alpha_{n}(D, \bm{a}; y) &= -\alpha_{n}(D_{-\epsilon(p)}, \bm{a}; y) + \alpha_{n+\Delta(p1)-1}(D_{p1}; y) + \alpha_{n+\Delta(p2)-1}(D_{p2}; y) \\
&= -\alpha_{n}(D_{-\epsilon(p)}, \bm{a}'; y) + \alpha_{n+\Delta(p1)-1}(D_{p1}; y) + \alpha_{n+\Delta(p2)-1}(D_{p2}; y) \\
&= \alpha_{n}(D, \bm{a}'; y).
\end{align*}

Now, let us consider the case where $\delta(p)=0$, $d_{\bm{a}}(D)=1$, and $d_{\bm{a}'}(D)=0$. Note that the number of crossing points of $D_{p1}$ and $D_{p2}$ is $m-1$. By taking a suitable sequence of base points and directions, $D_{p1}$ and $D_{p2}$ become monotone diagrams or connected sums of monotone diagrams. Using \eqref{equdef}, we have
\[
\alpha_{n}(D, \bm{a}; y) = -\alpha_{n}(D_{-\epsilon(p)}, \bm{a}; y) + \alpha_{n+\Delta(p1)-1}(D_{p1}; y) + \alpha_{n+\Delta(p2)-1}(D_{p2}; y).
\]
Since $\delta(p)=0$, either $\Delta(p1)=1$ and $\Delta(p2)=0$, or $\Delta(p1)=0$ and $\Delta(p2)=1$ holds. In either case, we obtain
\begin{align*}
&-\alpha_{n}(D_{-\epsilon(p)}, \bm{a}; y) + \alpha_{n+\Delta(p1)-1}(D_{p1}; y) + \alpha_{n+\Delta(p2)-1}(D_{p2}; y)\\
&\quad = -y^{w(D)-2\epsilon(p)}(-1)^n \binom{r-1}{n} (y+y^{-1})^{r-n-1} \\
&\quad\quad + y^{w(D)-\epsilon(p)}(-1)^n \binom{r}{n} (y+y^{-1})^{r-n} \\
&\quad\quad + y^{w(D)-\epsilon(p)}(-1)^{n-1} \binom{r-1}{n-1} (y+y^{-1})^{r-n} \\
&\quad = y^{w(D)}(-1)^{n-1}(y+y^{-1})^{r-n-1} \\
&\quad\quad \cdot \left\{ y^{-2\epsilon(p)} \binom{r-1}{n} - y^{-\epsilon(p)} \binom{r}{n}(y+y^{-1}) + y^{-\epsilon(p)} \binom{r-1}{n-1}(y+y^{-1}) \right\}.
\end{align*}
Since $\epsilon(p)=\pm 1$, by a straightforward computation, we have
\begin{align*}
&y^{-2\epsilon(p)} \binom{r-1}{n} - y^{-\epsilon(p)} \binom{r}{n}(y+y^{-1}) + y^{-\epsilon(p)} \binom{r-1}{n-1}(y+y^{-1})\\
&\quad = -\binom{r-1}{n}.
\end{align*}
Thus, we see that 
\[
\alpha_{n}(D, \bm{a}; y) = y^{w(D)}(-1)^{n}\binom{r-1}{n} (y+y^{-1})^{r-n-1}.
\]
\end{proof}

\begin{lemma}\label{Reidemeister}
When $c(D) \le m$, the Laurent polynomials $\alpha_{n}(D;y)$ satisfy
\begin{align*}
\alpha_{n}\left(\RIpos; y\right) &= y \alpha_{n}\left(\RIm; y\right), \\
\alpha_{n}\left(\RIneg; y\right) &= y^{-1} \alpha_{n}\left(\RIm; y\right).
\end{align*}
Moreover, they are invariant under Reidemeister moves of types \II~and \III.
\end{lemma}

\begin{proof}
Let $D'$ be a diagram obtained from $D$ by a Reidemeister move. First, we show that by choosing suitable base points, the warping crossing points of $(D, \bm{a})$ and those of $(D', \bm{a})$ are in one-to-one correspondence. 

If $D'$ is obtained by a Reidemeister move of type~I, by choosing a base point $a_i$ as shown in Figure~\ref{WarpingRI}, we obtain the desired one-to-one correspondence.

\begin{figure}[htbp]
    \centering
\includegraphics[width=0.8\linewidth]{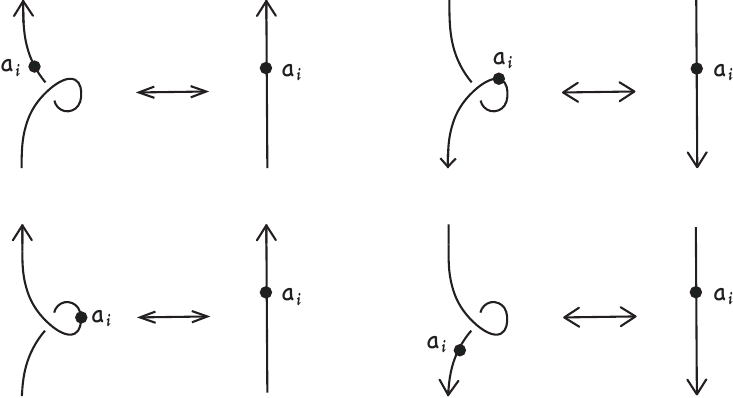}
    \caption{Correspondence of the sequence of base points under a Reidemeister move of type~I. The four possible orientation cases are shown.}
    \label{WarpingRI}
\end{figure}

Next, we show that when two strands belong to different components of $D$, $\alpha_n(D, \bm{a}; y)$ is invariant under the local transformation shown in Figure~\ref{CCRIIa}; that is, $\alpha_n(D, \bm{a}; y) = \alpha_n(D', \bm{a}; y)$. By \eqref{equdef}, we have
\begin{align*}
\alpha_n(D_{\epsilon(p)}, \bm{a}; y) &= -\alpha_n(D_{-\epsilon(p)}, \bm{a}; y) +  \alpha_{n+\Delta(p1)-1} (D_{p1}; y) +  \alpha_{n+\Delta(p2)-1} (D_{p2}; y).
\end{align*}

\begin{figure}[htbp]
    \centering
    \includegraphics[width=0.5\linewidth]{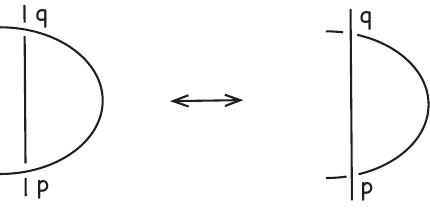}
    \caption{A local move relating diagrams $D$ (left) and $D'$ (right). The crossings $p$ and $q$ correspond under this move.}
    \label{CCRIIa}
\end{figure}

Again using \eqref{equdef}, we obtain
\begin{align*}
&-\alpha_n(D_{-\epsilon(p)}, \bm{a}; y) \\
&= \alpha_n(D', \bm{a}; y) - \alpha_{n+\Delta(q1)-1} ((D_{-\epsilon(p)})_{q1}; y) - \alpha_{n+\Delta(q2)-1} ((D_{-\epsilon(p)})_{q2}; y).
\end{align*}

From these equations, we have
\begin{equation}\label{eq:doubleRII}
\begin{split}
\alpha_n(D, \bm{a}; y) &= \alpha_n(D', \bm{a}; y) + \alpha_{n-1}(D_{p1}; y) + \alpha_{n-1}(D_{p2}; y)\\
&\quad - \alpha_{n-1}((D_{-\epsilon(p)})_{q1}; y) - \alpha_{n-1}((D_{-\epsilon(p)})_{q2}; y).
\end{split}
\end{equation}

For any orientation of the strands, by the induction hypothesis on the number of crossing points, we see that
\[
\alpha_n(D, \bm{a}; y) = \alpha_n(D',\bm{a}; y).
\]
\begin{figure}[htbp]
    \centering
\includegraphics[width=0.8\linewidth]{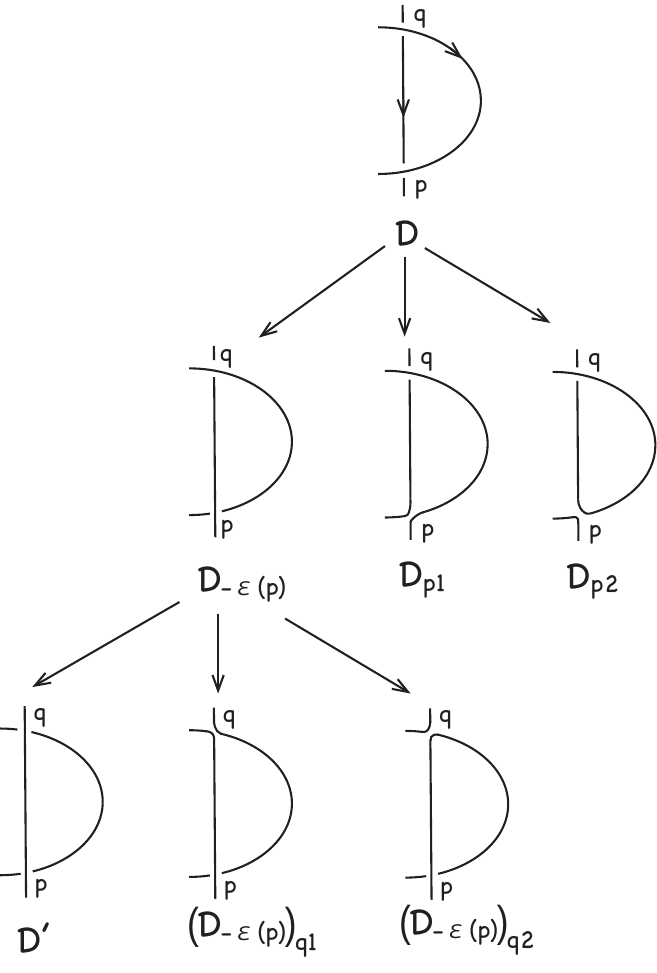}
    \caption{Correspondence of the six diagrams appearing in equation \eqref{eq:doubleRII}, obtained by applying the skein relation twice. The diagram $D_{-\epsilon(p)}$ is further expanded at the crossing $q$. The opposite orientation case, corresponding to a non-braid type bigon, is omitted.}
    \label{doubleRII}
\end{figure}
Figure~\ref{doubleRII} illustrates one case of the above calculation. For the other orientations, the same argument works. The details are left to the reader.  

Using this, we can adjust the upper and lower positions of the strands. By Lemma~\ref{warpingcrosspt}, setting the base point $a_i$ as shown in Figure~\ref{WarpingRII} yields a one-to-one correspondence.

\begin{figure}[htbp]
    \centering
    \includegraphics[width=1.0\linewidth]{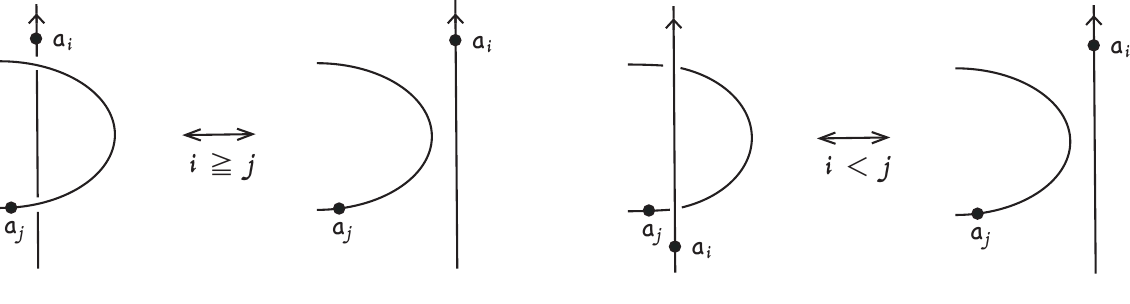}
    \caption{Correspondence of the sequence of base points under a Reidemeister move of type~\II. The two cases $i \ge j$ and $i < j$ are shown.}
    \label{WarpingRII}
\end{figure}
Moreover, for a Reidemeister move of type~\III, we have a one-to-one correspondence as shown in Figure~\ref{WarpingRIII}.

\begin{figure}[htbp]
    \centering
\includegraphics[width=0.5\linewidth]{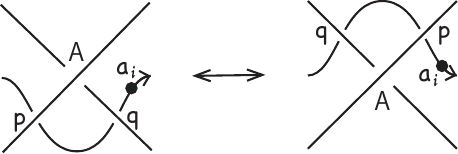}
    \caption{Correspondence of the sequence of base points under a Reidemeister move of type~\III. The crossings $p$ and $q$ correspond under the move.}
    \label{WarpingRIII}
\end{figure}

In the following, we proceed by induction on $s = d_{\bm{a}}(D) = d_{\bm{a}}(D')$. When $D$ is a monotone diagram or a connected sum of monotone diagrams, we have
\begin{align*}
\alpha_n(D,\bm{a};y) &= y^{w(D)}(-1)^{n} \binom{r-1}{n} (y+y^{-1})^{r-n-1},\\
\alpha_n(D',\bm{a};y) &= y^{w(D')}(-1)^{n} \binom{r-1}{n} (y+y^{-1})^{r-n-1}.
\end{align*}
From this, we see that
\begin{align*}
\alpha_{n}\left(\RIpos; y\right) &= y \alpha_{n}\left(\RIm; y\right), \\
\alpha_{n}\left(\RIneg; y\right) &= y^{-1} \alpha_{n}\left(\RIm; y\right).
\end{align*}
Since Reidemeister moves of types~\II~and \III~do not change the writhe, we have $\alpha_n(D,\bm{a};y)=\alpha_n(D',\bm{a};y)$.

Next, we assume that the result holds for $d_{\bm{a}_1}(D_1)=d_{\bm{a}_1}(D_1') < s$, where the warping crossing points of $D$ and $D'$ are assumed to be in one-to-one correspondence, as discussed above. Let $p$ be a corresponding warping crossing point of $(D, \bm{a})$ and $(D', \bm{a})$. Then we have
\begin{align*}
\alpha_n(D,\bm{a};y) &= -\alpha_n(D_{-\epsilon(p)}, \bm{a}; y) + \alpha_{n+\Delta(p1)-1}(D_{p1}; y) + \alpha_{n+\Delta(p2)-1}(D_{p2}; y),\\
\alpha_n(D',\bm{a};y) &= -\alpha_n(D'_{-\epsilon(p)}, \bm{a}; y) + \alpha_{n+\Delta(p1)-1}(D'_{p1}; y) + \alpha_{n+\Delta(p2)-1}(D'_{p2}; y).
\end{align*}

If $D'$ is the diagram obtained from $D$ by a Reidemeister move of type~I, by the induction hypothesis on $s$, we have
\[
\alpha_n(D_{-\epsilon(p)},\bm{a};y) = y^{\tau} \alpha_n(D'_{-\epsilon(p)},\bm{a};y).
\]
Here, $\tau\in\{+1,-1\}$ denotes the writhe change associated with the type~I move.  
If $D'$ is obtained by a Reidemeister move of types~\II~or \III, we have
\[
\alpha_n(D_{-\epsilon(p)},\bm{a};y) = \alpha_n(D'_{-\epsilon(p)},\bm{a};y).
\]

We distinguish two cases: first, the case where the chosen warping
crossing is not one of the three crossings appearing in
Figure~\ref{WarpingRIII}; second, the case where it is one of those
three crossings.

In the former case, $(D_{p1}, D'_{p1})$ and $(D_{p2}, D'_{p2})$ are
pairs of diagrams with at most $m-1$ crossings related by a
Reidemeister move, respectively.

By the induction hypothesis on the crossing number $m$, for a Reidemeister move of type~I, we have
\begin{align*}
\alpha_{n+\Delta(p1)-1}(D_{p1}; y) &= y^\tau \alpha_{n+\Delta(p1)-1}(D'_{p1}; y),\\
\alpha_{n+\Delta(p2)-1}(D_{p2}; y) &= y^\tau \alpha_{n+\Delta(p2)-1}(D'_{p2}; y).
\end{align*}

For Reidemeister moves of types~\II~and \III, we obtain
\begin{align*}
\alpha_{n+\Delta(p1)-1}(D_{p1}; y) &= \alpha_{n+\Delta(p1)-1}(D'_{p1}; y),\\
\alpha_{n+\Delta(p2)-1}(D_{p2}; y) &= \alpha_{n+\Delta(p2)-1}(D'_{p2}; y).
\end{align*}

In the latter case, let $u$ denote the chosen crossing in
Figure~\ref{WarpingRIII}.
Note that in Figure~\ref{WarpingRIII}, the crossings labeled $A$ and $q$
can be warping crossing points, whereas $p$ cannot.
We see that $(D_{u1}, D'_{u1})$ and $(D_{u2}, D'_{u2})$ are pairs of diagrams that are either identical or related by two Reidemeister moves of type~\II, having at most $m-1$ crossings. Thus, by the induction hypothesis on the crossing number $m$, we have
\begin{align*}
\alpha_{n+\Delta(u1)-1}(D_{u1}; y) &= \alpha_{n+\Delta(u1)-1}(D'_{u1}; y),\\
\alpha_{n+\Delta(u2)-1}(D_{u2}; y) &= \alpha_{n+\Delta(u2)-1}(D'_{u2}; y).
\end{align*}
This completes the proof.
\end{proof}

\begin{lemma}
For any diagram $D$ with $c(D)=m$, $\alpha_n(D, \bm{a}; y)$ does not depend on the choice of the sequence of base points $\bm{a}$.
\end{lemma}

\begin{proof}
If the number of components is $r=1$, we have $\bm{a}=(a_1)$, and any sequence of base points is connected to $\bm{a}$. Thus, the result holds by Lemma~\ref{basepoints}.

Next, assume that $r \ge 2$, and suppose that the result holds for diagrams with $r-1$ components.
We proceed by induction on $s = d_{\bm{a}}(D)$. If $s=0$, meaning $D$ is a monotone diagram, we can transform $D$ into $O+D'$ by Reidemeister moves of types~I, \II, and \III, where $D'$ is a monotone diagram. By Lemma~\ref{Reidemeister} and the induction hypothesis on $m$, we have
\begin{align*}
\alpha_n(D,\bm{a};y) &= y^{w(D)}(-1)^n \binom{r-1}{n} (y+y^{-1})^{r-n-1},\\
\alpha_n(D',\bm{a};y) &= y^{w(D')}(-1)^n \binom{r-2}{n} (y+y^{-1})^{r-n-2}.
\end{align*}
Hence, we obtain
\[
\alpha_n(D, \bm{a}; y) = y^{w(D)-w(D')} (y+y^{-1})\frac{r-1}{r-n-1} \alpha_n(D',\bm{a}';y).
\]
By the induction hypothesis on $r$, we see that $\alpha_n(D',\bm{a}'; y)$ does not depend on $\bm{a}'$. Thus, $\alpha_n(D, \bm{a}; y)$ does not depend on $\bm{a}$.

For the case $s>0$, by repeatedly applying the skein relation,
we can reduce $\alpha_n(D,\bm{a};y)$ to diagrams with fewer crossings,
and ultimately to the monotone case $s=0$.
Hence the result follows from the monotone case.
\end{proof}

Next, we show the following.

\begin{lemma}\label{propertyofalpha}
We can compute $\alpha_n (D; y)$ $(n=0, \pm 1, \dots)$ using (1)--(3) in Theorem~\ref{maintheorem1}. In particular, $\alpha_n(D; y) = 0$ except for finitely many $n$, and we have $\alpha_n(D; y) = 0$ for $n < 0$.
\end{lemma}

\begin{proof}
First, we show that for an $r$-component zero-crossing trivial link diagram $O^r$,  
\begin{equation}\label{equtrivial1}
\alpha_n (O^r;y)=
\begin{cases}
 (-1)^n \binom{r-1}{n} (y+y^{-1})^{r-n-1}, & \text{if $n \geq 0$,}\\
 0, & \text{if $n < 0$.}
\end{cases}
\end{equation}
Let $(D_+, D_-, D_{p1}, D_{p2})$ satisfy $c(D_+)=c(D_-)=1$. If $\delta(p)=1$, there exists another crossing, so $c(D_+)=c(D_-) \geq 2$. Thus, we have $\delta(p)=0$. Moreover, either $\Delta(p1)=1$ and $\Delta(p2)=0$, or $\Delta(p1)=0$ and $\Delta(p2)=1$. Using \eqref{equ1} and \eqref{equ3}, we obtain
\[
(y+y^{-1})\alpha_n (O^r; y) = \alpha_n(O^{r+1}; y) + \alpha_{n-1} (O^r;y).
\]
We can compute $\alpha_n(O^{r+1}; y)$ inductively using this equation. If $r=1$, by \eqref{equ2}, we have
\begin{align*}
\alpha_n(O^2; y) &= (y+y^{-1})\alpha_n(O; y) - \alpha_{n-1}(O; y)\\
 &= (y+y^{-1})\delta_{n,0} - \delta_{n-1,0}.
\end{align*}
This satisfies \eqref{equtrivial1}. Next, assume that \eqref{equtrivial1} holds for $O^r$, and consider the case of $O^{r+1}$. If $n \geq 0$, we have
\begin{align*}
\alpha_n(O^{r+1}; y) &= (-1)^n \binom{r-1}{n} (y+y^{-1})^{r-n} - (-1)^{n-1} \binom{r-1}{n-1}(y+y^{-1})^{r-(n-1)-1}\\
 &= (-1)^n \binom{r}{n}(y+y^{-1})^{r-n}.
\end{align*}
On the other hand, if $n < 0$, we have $\alpha_n (O^r; y)=0$ and $\alpha_{n-1} (O^r;y)=0$, and hence
\[
\alpha_n(O^{r+1}; y)=0.
\]
Thus, we see that equation \eqref{equtrivial1} holds for the zero-crossing trivial diagram $O^{r+1}$. Consequently, \eqref{equtrivial1} holds for any $n$ and $r$. From this and (1) in Theorem~\ref{maintheorem1}, if $D$ is monotone, then we have
\begin{equation}\label{equmonotone}
\alpha_n (D;y)=
\begin{cases}
 (-1)^n y^{w(D)} \binom{r-1}{n} (y+y^{-1})^{r-n-1}, & \text{if $n \geq 0$,}\\
 0, & \text{if $n < 0$,}
\end{cases}
\end{equation}
where $w(D)$ is the writhe.

Next, we show that $\alpha_n(D;y)$ can be computed for any diagram $D$ with $cd(D)=(k,s)$, assuming that $\alpha_n(D';y)$ can be computed for any diagram $D'$ with $cd(D')<(k,s)$. If $s=0$, $D$ is a monotone diagram, and $\alpha_n(D;y)$ satisfies equation~\eqref{equmonotone}. Therefore, we consider the case where $s>0$. We choose a sequence of base points $\bm{a}=(a_1,\dots,a_i,\dots, a_r)$ with $d_{\bm{a}}(D)=s$ and a warping crossing point $p$. Then there exist integers $s_1, s_2 \geq 0$ such that
\begin{align}\label{equcd}
cd(D_{-\epsilon(p)}) &\leq (k,s-1) < (k,s), \nonumber\\
cd(D_{p1}) &\leq (k-1, s_1) < (k,s), \\
cd(D_{p2}) &\leq (k-1, s_2) < (k,s). \nonumber
\end{align}
By the induction hypothesis, we can compute $\alpha_n (D_{-\epsilon(p)}; y)$, $\alpha_{n+\Delta(p1)-1}(D_{p1};y)$, and $\alpha_{n+\Delta(p2)-1}(D_{p2};y)$. Thus, we can calculate $\alpha_n(D;y)$ using \eqref{equ3}. 

Next, we show that $\alpha_n(D; y)=0$ for $n<0$. Assume that $\alpha_n(D'; y)=0$ for $n<0$ for any diagram $D'$ with $c(D')<k$. Let $D$ be a diagram with $c(D)=k$. For any crossing point $p$, we see that $c(D_{p1})<k$, $c(D_{p2})<k$, $n+\Delta(p1)-1<0$, and $n+\Delta(p2)-1<0$. Thus, by \eqref{equ3}, we have
\[
\alpha_n(D_+;y) + \alpha_n (D_-; y) = 0,
\]
where $D=D_+$ or $D=D_-$. By applying crossing changes repeatedly, the diagram $D$ can be transformed into a monotone diagram. Therefore, we conclude that $\alpha_n(D; y)=0$ for $n<0$.

Finally, using \eqref{equ3} and \eqref{equcd}, and by induction on $cd(D)=(k,s)$, we see that $\alpha_n(D;y)=0$ except for finitely many $n$.
\end{proof}



From Theorem~\ref{maintheorem1}, the generating series $L$ associated
with the coefficient polynomials $\alpha_n$ and its writhe-normalized
version $F$ satisfy the following properties.

\begin{prop}\label{prop:Lskein}
If $D$ and $D'$ are regularly isotopic diagrams, then $L_D=L_{D'}$. Moreover, the following identities hold for all skein quadruples of diagrams that are identical everywhere except in the local pictures indicated below:
\begin{align*}
& L_{D_+} + L_{D_-} = z( L_{D_{\infty}} + L_{D_0}), \\
&L_{\includegraphics{RIpos.pdf}} = yL_{\includegraphics{RIm.pdf}},\\
&L_{\includegraphics{RIneg.pdf}}= y^{-1}L_{\includegraphics{RIm.pdf}}, \\
&L_{\bigcirc}=1.
\end{align*}
\end{prop}

\begin{prop}\label{prop:Fambient}
The Laurent polynomial $F_D(y,z)\in R[z,z^{-1}]$ is an ambient isotopy invariant of oriented links.
\end{prop}

Lemma~\ref{lem:connected} verifies that the polynomial $L$ behaves multiplicatively under connected sum and satisfies the expected formula for disjoint union.

\begin{lemma}\label{lem:connected}
The following formulas hold for the polynomial $L$ with respect to connected sum and disjoint union:
\[
L_{D \# D'} = L_D L_{D'}, \quad L_{D\sqcup D'} = d L_D L_{D'},
\]
where $d = z^{-1}(y+y^{-1})-1$.
\end{lemma}

\begin{proof}
First, we consider the case where $D$ and $D'$ are monotone diagrams or connected sums of monotone diagrams with $r$ and $r'$ components, respectively. Then we have
\begin{align*}
L_D &= z^{1-r} \sum_{n=0}^\infty \alpha_n(D;y) z^n \\
    &= z^{1-r} \sum_{n=0}^{\infty} y^{w(D)} (-1)^{n} \binom{r-1}{n} (y + y^{-1})^{r-n-1} z^{n}, \\
L_{D'} &= z^{1-r'} \sum_{n=0}^\infty \alpha_n(D';y) z^n \\
       &= z^{1-r'} \sum_{n=0}^{\infty} y^{w(D')} (-1)^{n} \binom{r'-1}{n} (y + y^{-1})^{r'-n-1} z^{n}.
\end{align*}
Hence, we obtain
\begin{align*}
L_D L_{D'} &= z^{2-r-r'} \sum_{n=0}^\infty \left( \sum_{k=0}^n \alpha_k(D;y)\alpha_{n-k}(D';y) \right) z^n \\
&= z^{2-r-r'} \sum_{n=0}^{\infty} \Biggl( \sum_{k=0}^{n} y^{w(D)} (-1)^{k} \binom{r-1}{k} (y + y^{-1})^{r-k-1} \\
&\qquad \cdot y^{w(D')} (-1)^{n-k} \binom{r'-1}{n-k} (y + y^{-1})^{r'-(n-k)-1} \Biggr) z^{n} \\
&= z^{2-r-r'} \sum_{n=0}^{\infty} y^{w(D) + w(D')} (-1)^{n} (y + y^{-1})^{r+r'-n-2} \\
&\qquad \cdot\left( \sum_{k=0}^{n} \binom{r-1}{k} \binom{r'-1}{n-k} \right) z^{n} \\
&= z^{2-r-r'} \sum_{n=0}^{\infty} y^{w(D \# D')} (-1)^{n} \binom{r+r'-2}{n} (y + y^{-1})^{r+r'-n-2} z^{n} \\
&= L_{D \# D'},
\end{align*}
by Definition~\ref{defofalpha}.

Next, we assume that $L_{D \# D'} = L_D L_{D'}$ whenever $cd(D) < (k,s)$. Let $p$ be a warping crossing point of $D$. From Proposition~\ref{prop:Lskein}, we have
\[
L_{D \# D'} = -L_{D_{-\epsilon(p)} \# D'} + z L_{D_\infty \# D'} + z L_{D_0 \# D'}.
\]
Since $cd(D_{-\epsilon(p)}) < (k,s)$, $cd(D_\infty) < (k,s)$, and $cd(D_0) < (k,s)$, by the induction hypothesis, we have
\begin{align*}
L_{D \# D'} &= (-L_{D_{-\epsilon(p)}} + z L_{D_\infty} + z L_{D_0}) L_{D'} \\
&= L_D L_{D'}.
\end{align*}

We next show $L_{D \sqcup D'} = d L_D L_{D'}$ when $D$ and $D'$ are monotone diagrams or connected sums of monotone diagrams with $r$ and $r'$ components, respectively. Using the computation above, we obtain
\begin{align*}
&d L_D L_{D'} \\
&= z^{2-r-r'} \left( \sum_{n=0}^{\infty} y^{w(D)+w(D')} (-1)^{n} (y + y^{-1})^{r+r'-n-2} \binom{r+r'-2}{n} z^{n} \right) \\
&\qquad \cdot \left( z^{-1}(y+y^{-1}) - 1 \right) \\
&= z^{1-r-r'} \left( \sum_{n=0}^\infty y^{w(D)+w(D')} (-1)^n (y+y^{-1})^{r+r'-n-1} \binom{r+r'-2}{n} z^n \right) \\
&\quad + z^{1-r-r'} \left( \sum_{n=0}^\infty y^{w(D)+w(D')} (-1)^{n+1} (y+y^{-1})^{r+r'-n-2} \binom{r+r'-2}{n} z^{n+1} \right) \\
&= z^{1-r-r'} \Biggl( \sum_{n=1}^\infty y^{w(D)+w(D')} (-1)^n (y+y^{-1})^{r+r'-n-1} \binom{r+r'-2}{n} z^n \\
&\quad + y^{w(D)+w(D')} (y+y^{-1})^{r+r'-1} \binom{r+r'-2}{0} z^0 \Biggr) \\
&\quad + z^{1-r-r'} \left( \sum_{n=1}^\infty y^{w(D)+w(D')} (-1)^{n} (y+y^{-1})^{r+r'-n-1} \binom{r+r'-2}{n-1} z^{n} \right) \\
&= z^{1-r-r'} \Biggl( \sum_{n=1}^\infty y^{w(D)+w(D')} (-1)^n (y+y^{-1})^{r+r'-n-1} \binom{r+r'-1}{n}  z^n \\
&\qquad + y^{w(D)+w(D')} (y+y^{-1})^{r+r'-1} \Biggr) \\
&= z^{1-r-r'} \sum_{n=0}^\infty y^{w(D)+w(D')} (-1)^n (y+y^{-1})^{r+r'-n-1} \binom{r+r'-1}{n} z^n \\
&= L_{D \sqcup D'}.
\end{align*}
The general case of the disjoint union formula is proved by the same induction on $cd(D)$ as in the connected-sum case above.
\end{proof}

\section{Uniqueness}\label{sec:uniqueness}
\subsection{Uniqueness under regular isotopy}\label{subsec:uniq_regular}
The uniqueness for three-term skein relations is well known
(see, e.g., \cite{ItoYoshidaNakagane}).
In contrast, the case of four-term skein relations,
especially for invariants under regular isotopy,
requires a more careful analysis.
For ambient isotopy invariants, the argument proceeds in parallel with the classical case.
On the other hand, for regular isotopy invariants, the behavior under Reidemeister moves of type~I
introduces additional complications, and the inductive argument must be handled more carefully.
In this section, we study the uniqueness in the four-term skein setting,
both for regular and ambient isotopy.
While our argument follows the general strategy of the three-term case,
it incorporates essential modifications reflecting the four-term structure.
As a consequence, we show that the construction developed above,
in particular the coefficient polynomials $\alpha_n$,
directly leads to a reconstruction of the Kauffman polynomial.
\begin{lemma}\label{lem:uniq_regular}
Let $R$ be a commutative ring containing $\mathbb{Z}[y^{\pm1}]$
and in which $a,b,c,d\in R$ are invertible.
Let $f$ and $f'$ be two link invariants with values in $R$
satisfying the same four-term skein relation
\begin{equation}\label{eq:4term_general}
a\,f(L_+) + b\,f(L_-) + c\,f(L_\infty) + d\,f(L_0)=0
\end{equation}
for every skein quadruple $(L_+,L_-,L_\infty,L_0)$,
together with the normalization
\begin{equation}\label{eq:trivial_knot_agree}
f(O)=f'(O)
\end{equation}
for the zero-crossing trivial knot diagram $O$,
and the kink relations
\begin{equation}\label{eq:kink_relations_general}
f(\RIpos)=y\,f(\RIm), \qquad
f(\RIneg)=y^{-1}f(\RIm)
\end{equation}
(and the same for $f'$).
Then $f\equiv f'$ as invariants under regular isotopy.
\end{lemma}

\begin{proof}
Let $F=f-f'$. It suffices to prove $F(L)=0$ for every link $L$.

\medskip
\noindent\emph{Step 1: zero-crossing unlink diagrams.}
From \eqref{eq:trivial_knot_agree} we have $F(O)=0$.
Using \eqref{eq:4term_general} together with the kink relations
\eqref{eq:kink_relations_general},
one can express $F(O^{r+1})$ linearly in terms of $F(O^r)$.
(Concretely, apply the skein relation to a one-crossing diagram whose smoothings
are $O^r$ and $O^{r+1}$; since $c$ and $d$ are invertible in $R$,
one can solve for the value on one smoothing in terms of the other.)
Hence, by induction on $r\ge 1$, we obtain
\[
F(O^r)=0 \qquad \text{for all zero-crossing unlink diagrams}~O^r.
\]
If a diagram $D$ of a trivial link is connected to $O^r$ by a finite sequence of Reidemeister moves,
then the kink relations together with regular isotopy invariance imply $F(D)=0$.
Hence $F(L)=0$ for every trivial link $L$.    

\medskip
\noindent\emph{Step 2: minimal counterexample.}
Assume for contradiction that there exists a nontrivial link $L$ with $F(L)\ne 0$.
Choose such an $L$ with minimal crossing number $c(L)$, and fix a diagram $D$
realizing $c(D)=c(L)$.

It is standard that by changing crossings of $D$ finitely many times
one obtains a diagram of a trivial link.
Let
\[
D=D^{(0)},D^{(1)},\dots,D^{(r)}
\]
be a sequence of diagrams obtained from $D$ by successive crossing changes,
so that $D^{(r)}$ represents a trivial link.
Let $L^{(i)}$ be the link represented by $D^{(i)}$.
By construction,
\[
c(D^{(i)})=c(D)=c(L)\qquad\text{for all }i.
\]

Fix $i$ and apply the skein relation \eqref{eq:4term_general}
at the crossing that is changed from $D^{(i)}$ to $D^{(i+1)}$.
The two smoothing diagrams have strictly fewer crossings:
\[
c\big((D^{(i)})_\infty\big),\ c\big((D^{(i)})_0\big)\le c(D^{(i)})-1=c(L)-1.
\]
Therefore the corresponding links have crossing number $<c(L)$, and by the minimality of $c(L)$
we must have
\[
F\big((L^{(i)})_\infty\big)=F\big((L^{(i)})_0\big)=0.
\]
Since $a$ and $b$ are invertible in $R$, the skein relation reduces to an equivalence
\[
F(L^{(i)})=0 \iff F(L^{(i+1)})=0,
\]
and hence
\[
F(L^{(i)})\ne 0 \iff F(L^{(i+1)})\ne 0.
\]
Consequently $F(L^{(r)})\ne 0$.
But $L^{(r)}$ is a trivial link, so Step~1 gives $F(L^{(r)})=0$.
This contradiction proves $F\equiv 0$, hence $f\equiv f'$.
\end{proof}

\begin{remark}\label{rem:uniq_apply_to_L_and_alpha}
Lemma~\ref{lem:uniq_regular} applies to regular isotopy invariants
$L_D(y,z)\in R[[z]]$ (in particular $R=\mathbb{Z}[y^{\pm1}][[z]]$).
Moreover, coefficient extraction $[z^n]$ is an $R$-linear map on $R[[z]]$,
so the same uniqueness applies coefficientwise:
if $L_D=L'_D$, then $\alpha_n(D;y)=\alpha'_n(D;y)$ for all $n\ge 0$.
\end{remark}

\subsection{Uniqueness under ambient isotopy}\label{subsec:uniq_ambient}

\begin{prop}\label{prop:uniq_ambient}
Let $L_D(y,z)$ be a regular isotopy invariant taking values in $R[[z]]$
and satisfying the kink relations \eqref{eq:kink_relations_general}.
Define the writhe-normalized invariant
\[
F_D(y,z):=y^{-w(D)}\,L_D(y,z)
\]
for oriented diagrams $D$.
Then $F_D$ is invariant under ambient isotopy.
Moreover, if $L$ and $L'$ satisfy the same skein relation and the same normalization on the unknot,
then their writhe-normalizations coincide under ambient isotopy.
\end{prop}

\begin{proof}
Reidemeister moves of type~\II~and type~\III~do not change the writhe,
so $y^{-w(D)}$ is unchanged and the regular isotopy invariance of $L_D$
implies invariance of $F_D$ under type~\II~and type~\III.
Under a type~I move, $w(D)$ changes by $\pm 1$ and $L_D$ changes by a factor of $y^{\pm 1}$
by \eqref{eq:kink_relations_general}, so the product $y^{-w(D)}L_D$ is unchanged.
Hence $F_D$ is an ambient isotopy invariant.

For the uniqueness statement, Lemma~\ref{lem:uniq_regular} implies $L\equiv L'$ under regular isotopy.
Applying the same writhe normalization to both yields $F\equiv F'$ under ambient isotopy.
\end{proof}
Assume that $y=1$, we have Corollary~\ref{cor:y1uniq}.  
\begin{cor}\label{cor:y1uniq}
Let $f$ and $f'$ be two link invariants with values in a commutative ring $R$
such that
\[
a\,f(L_+) + b\,f(L_-) + c\,f(L_\infty) + d\,f(L_0)=0
\]
for every skein quadruple $(L_+,L_-,L_\infty,L_0)$,
where $a,b,c,d\in R$ are invertible.
If
\[
f(U)=f'(U)
\]
for the unknot $U$,
then $f\equiv f'$.
\end{cor}

\begin{proof}
Let $F=f-f'$. It suffices to show that $F(L)=0$ for every link $L$.

Since $f(U)=f'(U)$, we have $F(U)=0$.
Let $L_+$ and $L_-$ be the knots represented by the positive and negative one-crossing knot diagrams,
let $L_0$ be the link represented by the zero-crossing $2$-component diagram,
and let $L_\infty$ be the knot represented by the zero-crossing knot diagram.
Since $L_+$, $L_-$, and $L_\infty$ all represent the unknot $U$, we obtain
\[
F(L_0)=-c^{-1}(a+b+d)F(U)=0.
\]
Repeating the same argument inductively, we obtain
\[
F(U_{\ell})=0
\qquad\text{for every trivial link } U_{\ell}.
\]

Assume for contradiction that there exists a link $L$ such that $F(L)\neq 0$.
Choose such an $L$ with minimal crossing number $c(L)$, and fix a diagram $D$
realizing $c(D)=c(L)$.

By changing crossings of $D$ finitely many times,
we obtain a diagram of a trivial link.
Let
\[
D=D^{(0)},D^{(1)},\dots,D^{(r)}
\]
be the sequence of diagrams obtained by successive crossing changes,
and let $L^{(i)}$ be the link represented by $D^{(i)}$.
Then $L^{(r)}$ is a trivial link.

For each $i$, applying the skein relation at the crossing changed from $D^{(i)}$ to $D^{(i+1)}$,
we obtain
\[
a\,F(L^{(i)}) + b\,F(L^{(i+1)}) + c\,F((L^{(i)})_\infty) + d\,F((L^{(i)})_0)=0.
\]
Since smoothing reduces the crossing number,
\[
c((L^{(i)})_\infty)<c(L), \qquad c((L^{(i)})_0)<c(L),
\]
and by minimality of $c(L)$ we have
\[
F((L^{(i)})_\infty)=F((L^{(i)})_0)=0.
\]
Hence
\[
F(L^{(i+1)})=-b^{-1}a\,F(L^{(i)}),
\]
so in particular
\[
F(L^{(i)})\neq 0 \iff F(L^{(i+1)})\neq 0.
\]
Therefore $F(L^{(r)})\neq 0$.
But $L^{(r)}$ is a trivial link, so $F(L^{(r)})=0$, a contradiction.
Thus $F\equiv 0$, and hence $f\equiv f'$.
\end{proof}
\subsection{A direct uniqueness proof for the coefficient polynomials $\alpha_n$}\label{subsec:uniq_direct_alpha}

\begin{remark}\label{rem:direct_alpha}
The uniqueness of the coefficient polynomials $\alpha_n(D;y)$
can also be shown directly from the inductive construction
based on the complexity $cd(D)=(c(D),d_{\bm a}(D))$,
without invoking Lemma~\ref{lem:uniq_regular}.
We include a proof sketch for completeness.
\end{remark}

\begin{proof}[Sketch of direct proof]
First determine $\alpha_n(O^r;y)$ for trivial link diagrams $O^r$.
Using the skein relation together with the normalization $\alpha_n(O;y)=\delta_{n,0}$,
one obtains the recurrence
\[
(y+y^{-1})\alpha_n(O^r;y)
=
\alpha_n(O^{r+1};y)+\alpha_{n-1}(O^r;y).
\]
This uniquely determines
\[
\alpha_n(O^r;y)
=
\begin{cases}
(-1)^n {r-1\choose n}(y+y^{-1})^{r-n-1}, & n\ge0,\\
0, & n<0.
\end{cases}
\]

Next proceed by induction on $cd(D)=(k,s)$.
If $s=0$, then $D$ is monotone and the above formula determines $\alpha_n(D;y)$ uniquely.
Assume $s\ge 1$ and that $\alpha_n$ is uniquely determined for diagrams with smaller $cd$.
Let $p$ be a warping crossing of $(D, \bm{a})$. Then
\[
\alpha_n(D;y)
=
-\alpha_n(D_{-\epsilon(p)};y)
+
\alpha_{n+\Delta(p1)-1}(D_{p1};y)
+
\alpha_{n+\Delta(p2)-1}(D_{p2};y),
\]
and each term on the right-hand side has strictly smaller $cd$.
By induction, each term is uniquely determined, hence so is $\alpha_n(D;y)$.
\end{proof}

Combining Remark~\ref{rem:direct_alpha} with  Lemma \ref{propertyofalpha}, we obtain the following uniqueness statement.  
\begin{prop}
We can uniquely compute $\alpha_n (D;y)$ $(n=0, \pm 1, \dots)$ using (1)--(3) in Theorem \ref{maintheorem1}. In particular, $\alpha_n(D; y) = 0$ except for finitely many $n$, and $\alpha_n(D; y) = 0$ for $n < 0$.
\end{prop}

\section{Discussion}\label{sec:discussion}

We have established the existence of B-type coefficient polynomials $\alpha_n(D;y)$
within the framework of formal power series,
extending Kawauchi’s A-type construction
to the intrinsically four-term skein setting.
The essential new phenomenon in the B-type case
is that the inductive closure cannot be achieved
within the class of monotone diagrams alone.
The necessary enlargement to connected sums of monotone diagrams
reflects a genuine structural difference
between three-term and four-term skein theories.

From this viewpoint,
the generating series of the B-type coefficient polynomials should be regarded
not merely as a reconstruction of the Kauffman polynomial \cite{Kauffman1990},
but as a coefficient-level refinement
of the four-term skein structure.
Each coefficient polynomial defines a link invariant,
yielding an infinite sequence of invariants
associated naturally with the B-type skein relation.

The present paper focuses on the construction
and well-definedness of this theory.
Further investigations,
including structural comparisons with mutation phenomena
and interactions with random knot models,
will be pursued elsewhere.

\bibliographystyle{amsplain}
\bibliography{BtypeRef}

\end{document}